\documentclass[11pt]{extarticle}

\usepackage{amsmath}
\usepackage{amsthm}
\usepackage{relsize}

\usepackage{comment}
\usepackage{epsfig}
\usepackage{float}
\usepackage{hyperref}
\hypersetup{
     colorlinks   = true,
     citecolor    = black,
     linkcolor    = black
}
\usepackage{textcomp}
\usepackage{amsbsy}
\usepackage{latexsym}
\usepackage[mathscr]{eucal}
\usepackage{amsfonts}
\usepackage{amssymb}
\usepackage[usenames]{color}
\usepackage{easybmat}
\usepackage{soul}   % for \st  

\usepackage[utf8]{inputenc}   %package for accents

\newcommand{\bD}{\mathbf D}
\newcommand{\bG}{\mathbf G}

\newcommand{ \cF}{{\cal F}}
\newcommand{\Nset}{\mathbb{N}}
\newcommand{\Rset}{\mathbb{R}}
\newcommand{\support}{\Gamma}

\newcommand{\genpol}{v}

\newcommand{\tarfun}{u}

\newcommand{\indmeas}{i}
\newcommand{\inddim}{j}
\newcommand{\indbasis}{k}
\newcommand{\truncation}{\tau}

\def\<{\langle}
\def\>{\rangle}

\newcommand{\be}{\begin{equation}}
\newcommand{\ee}{\end{equation}}
\newcommand{\beu}{\begin{equation*}}
\newcommand{\eeu}{\end{equation*}}

\newtheorem{lemma}{Lemma}
\newtheorem{theorem}{Theorem}

\newtheorem{remark}{Remark}

\DeclareMathOperator*{\argmin}{argmin}
\DeclareMathOperator*{\supp}{supp}

\newcommand{\E}{\mathbb{E}}
\renewcommand{\P}{\mathbb{P}}
\newcommand{\cA}{{\cal A}}
\newcommand{\cM}{{\cal M}}
\newcommand{\cR}{{\cal R}}
\newcommand{\cH}{{\cal H}}

\newcommand{\bI}{{\bf I}}
\newcommand{\bb}{{\bf b}}
\newcommand{\ba}{{\bf a}}

\newcommand{\paraone}{\theta_1}
\newcommand{\paratwo}{\theta_2}

\newcommand{\costgen}{C}

\def\t{\tilde}

\def\wt{\widetilde}
\def\wh{\widehat}

\def\cA{{\cal A}}

\def\cR{{\cal R}}

\def\cJ{{\cal J}}
\def\cM{{\cal M}}

\def\[{\Bigl [}
\def\]{\Bigr ]}
\def\({\Bigl (}
\def\){\Bigr )}
\def\[{\Bigl [}
\def\]{\Bigr ]}
\def\({\Bigl (}
\def\){\Bigr )}

\begin{document}

\setcounter{page}{1}
\title
{
Discrete least-squares approximations over optimized downward closed polynomial spaces in arbitrary dimension
\thanks
{ 
This research was supported by the Institut Universitaire de France, the ERC AdG BREAD, and the center for advanced modeling science (CADMOS).
}
}
\author{
Albert Cohen\thanks{Sorbonne Universit\'es, UPMC Univ Paris 06, CNRS, UMR 7598, Laboratoire Jacques-Louis Lions, 4, place Jussieu, Paris 75005, France. email: cohen@ann.jussieu.fr}
\and 
Giovanni Migliorati\thanks{Sorbonne Universit\'es, UPMC Univ Paris 06, CNRS, UMR 7598, Laboratoire Jacques-Louis Lions, 4, place Jussieu, Paris 75005, France. email: migliorati@ann.jussieu.fr} 
\and 
Fabio Nobile\thanks{MATHICSE-CSQI, \'Ecole Polytechnique F\'ed\'erale de Lausanne, Lausanne CH-1015, Switzerland. email: fabio.nobile@epfl.ch} 
}
\date{
\today
}

\maketitle

\begin{abstract}
\noindent
We analyze the accuracy of the discrete least-squares approximation of a function $\tarfun$ in multivariate polynomial spaces $\P_\Lambda:={\rm span} \{y\mapsto y^\nu \,: \, \nu\in \Lambda\}$ with $\Lambda\subset \Nset_0^d$ over the domain $\Gamma:=[-1,1]^d$, based on the sampling of this function at points $y^1,\dots,y^m \in \Gamma$. The samples are independently drawn according to a given probability density $\rho$ belonging to the class of multivariate beta densities, which includes the uniform and Chebyshev densities as particular cases. Motivated by recent results on high-dimensional parametric and stochastic PDEs, we restrict our attention to polynomial spaces associated  with \emph{downward closed} sets $\Lambda$ of \emph{prescribed} cardinality $n$, and we optimize the choice of the space for the given sample. This implies, in particular, that the selected polynomial space depends on the sample. We are interested in comparing the error of this least-squares approximation measured in $L^2(\Gamma,d\rho)$ with the best achievable polynomial approximation error when using downward closed sets of cardinality $n$. We establish conditions between the dimension $n$ and the size $m$ of the sample, under which these two errors are proven to be comparable. Our main finding is that the dimension $d$ enters only moderately in the resulting trade-off between $m$ and $n$, in terms of a logarithmic factor $\ln(d)$, and is even absent when the optimization is restricted to a relevant subclass of downward closed sets, named {\it anchored} sets. In principle, this allows one to use these methods in arbitrarily high or even infinite dimension. Our analysis builds upon \cite{CCMNT2013} which considered fixed and nonoptimized downward closed multi-index sets. Potential applications of the proposed results are found in the development and analysis of efficient numerical methods for computing the solution to high-dimensional parametric or stochastic PDEs, albeit not limited to this area.
\end{abstract}
\bigskip


\section{Introduction}
\label{sec:introduction}
\noindent
In recent years it has become clear that many interesting engineering applications feature an intrinsic dependence on a large number of parameters $y_1,\dots,y_d$, leading to a major concentration of efforts in the development and analysis of high-dimensional approximation methods. In many relevant situations, the parameters $y_\inddim$ are independent random variables distributed on intervals $I_\inddim$ according to probability measures $d\rho_\inddim$. We are then typically interested in approximating a function
\beu
y=(y_1,\dots,y_d)\mapsto \tarfun(y),
\eeu
depending on these parameters and measuring the error in $L^2(\Gamma,d\rho)$, where $\Gamma=I_1\times \cdots \times I_d$ and $d\rho=d\rho_1 \otimes \cdots \otimes d\rho_d$.
Up to a renormalization, we may assume that $I_\inddim=[-1,1]$ for all $\inddim$, so that $\Gamma=[-1,1]^d$. In certain situations, the number of parameters may even be countably infinite, in which case $\Gamma=[-1,1]^{\Nset}$.
Examples where such problems occur are recurrent in the numerical treatment of parametric and stochastic PDEs, where a fast and accurate approximation of the parameter-to-solution map over high-dimensional parameter sets is useful to tackle more complex optimization, control and inverse problems. 

In such a context, the potential of specific high-dimensional \emph{polynomial} approximation methods has been demonstrated in \cite{CDS2011,CD2015,MNVT2013,Migliorati2013,CCMNT2013}. In these methods, the approximation is picked from a multivariate polynomial space
\beu
\P_\Lambda:={\rm span} \{y \mapsto y^\nu \,: \, \nu\in \Lambda\},
\eeu
where $\Lambda$ is a given finite subset of $\Nset_0^d$. In the case of countably many parameters, $d=\infty$, we replace $\Nset_0^d$ by the set of finitely supported sequences of nonnegative integers.

The set $\Lambda$ is said to be \emph{downward closed} if and only if
\be
\nu \in \Lambda \quad {\rm and} \quad \mu\leq \nu \implies  \mu\in \Lambda,
\label{dc}
\ee
where $\mu\leq \nu$ is meant component-wise as $\mu_i\leq \nu_i$ for all $i$. Polynomial spaces $\P_\Lambda$ associated to downward closed index sets $\Lambda$ have been studied in various contexts, see \cite{CCS,DF,DR,K,LL}.

There exist two main approaches to polynomial approximation of a given function $u$ based on pointwise evaluations. The first approach relies on \emph{interpolation} of the function $\tarfun$ at a given set of points $\{y^1,\dots,y^n\}$ where $n:=\#(\Lambda)=\dim(\P_\Lambda)$, that is, find $\genpol \in \P_\Lambda$ such that $\genpol(y^\indmeas)=\tarfun(y^\indmeas)$ for $\indmeas=1,\dots,n$. 
The second approach relies on \emph{projection}, which aims at minimizing the $L^2(\Gamma,d\rho)$ error between $\tarfun$ and its approximation in $\P_\Lambda$. Since the exact projection is not available, one typical strategy consists in using the discrete least-squares method, that is, solving the problem
\beu
\min_{\genpol \in \P_\Lambda} \sum_{\indmeas=1}^m |\genpol(y^\indmeas)-\tarfun(y^\indmeas)|^2,
\eeu
where now $m>n$. 
Discrete least-squares methods are often preferred to interpolation methods when the observed evaluations are polluted by noise. Their convergence analysis has been studied in the general context of learning theory, see for example \cite{CZ2007,PS2003,Gyorfi2002,T2008,Vapnik1998}. 

In recent years, an analysis of discrete least-squares methods has been proposed \cite{CCMNT2013,MNST2011,Migliorati2013,MNT2015}, specifically targeted to the above described case of multivariate polynomial spaces associated with downward closed sets,
in the case where the $d\rho_\inddim$ are identical Jacobi-type measures. 
This analysis, which builds upon the general results from \cite{CDL2013}, gives conditions ensuring that, in the absence of noise in the pointwise evaluation of $u$, the accuracy of the discrete least-squares approximation is comparable to
the best approximation error achievable in $\P_\Lambda$, that is,
\beu
e_{\Lambda}(\tarfun):=\inf_{\genpol \in \P_\Lambda} \|\tarfun - \genpol \|_{L^2(\Gamma,d\rho)}.
\eeu
These conditions are stated in terms of a relation between the size $m$ of the sample and the dimension $n$ of $\P_\Lambda$. A similar analysis also covers the case of an additive noise in the evaluation of the samples, which results in additional terms in the error estimate, see \emph{e.g.} \cite{MNT2015}.

One remarkable result from the above analysis is that the conditions ensuring that the least-squares method has accuracy comparable to $e_{\Lambda}(\tarfun)$ only involve the dimension of $\P_\Lambda$. These conditions are independent of the specific shape of the set $\Lambda$ (as long as it is downward closed), and in particular independent of the dimension $d$. 

The possibility of using arbitrary sets $\Lambda$ is critical in the context of parametric PDEs in view of the recent results on high-dimensional polynomial approximation obtained in  \cite{CDS2011,CCS2013,CD2015}.
These results show that for relevant classes of parametric PDEs, the functions $y\mapsto u(y)$ describing either the full solution or scalar quantities of interest can be approximated with convergence rates ${\cal O}(n^{-s})$ which are {\it independent} of the parametric dimension $d$, when using polynomial spaces $\P_{\Lambda_n}$ associated to \emph{specific} sequences of downward closed multi-index sets $(\Lambda_n)_{n \geq 1}$ with $\#(\Lambda_n)=n$. 
In summary, we have
\be
e_n(\tarfun):=\min_{\#(\Lambda)=n} e_{\Lambda}(\tarfun) \leq Cn^{-s},
\label{opterror}
\ee
where the minimum is taken over all downward closed sets of given cardinality $n$.

For each value of $n$, the optimal set $\Lambda_n$ is the one that achieves the minimum in \eqref{opterror} among all downward closed $\Lambda$ of cardinality $n$. 
This set is unknown to us when observing only the samples $\tarfun(y^\indmeas)$. 
Therefore, a legitimate objective is to develop least-squares methods for which the accuracy is comparable to the quantity $e_n(\tarfun)$.

In this paper, we discuss least-squares approximations on multivariate polynomial spaces for which the choice of $\Lambda$ is optimized based on the sample.
In particular we prove that the performance of such approximations is comparable to the quantity in \eqref{opterror}, under a relation between $m$ and $n$ where the dimension $d$ enters as a logarithmic factor. 
Furthermore, we show that this logarithmic dependence on $d$ can be fully removed by considering a more restricted class of downward closed sets called \emph{anchored sets}, for which similar approximation rates as in \eqref{opterror} can be achieved. The resulting least-squares methods are thus immune to the curse of dimensionality.

The outline of the paper is the following: in Section~\ref{sec:ls_resume} we introduce the notation and
briefly review some of the previous results achieved in the analysis of discrete least squares on {\em fixed} multivariate polynomial spaces. 
In Section~\ref{sec:bestnterm_downward_closed} we present the main results of the paper concerning discrete
least-squares approximations on {\em optimized} polynomial spaces. Our analysis is based on
establishing upper bounds on the number of downward closed 
or anchored sets of a given cardinality, or on the cardinality of their union. 

The selection of the optimal polynomial space is based on minimizing
the least-squares error among all possible choices of downward closed or anchored sets of a given cardinality $n$.
Let us stress that in the form of an exhaustive search, this task becomes computationally intensive when $n$ and $d$ are
simultaneously large.
Our results should therefore mainly be viewed as a benchmark in arbitrary dimension $d$ 
for assessing the performance of fast selection algorithms,
such as greedy algorithms, that still need to be developed and analyzed in this context. 
A general discussion on alternate selection strategies with reasonable computational cost is presented in the final section \S 4.

\section{Discrete least-squares approximation by multivariate polynomials}
\label{sec:ls_resume}
In this section we introduce some useful notation, and recall from \cite{CCMNT2013} the main results achieved for the analysis of the stability and accuracy of discrete least-squares approximations in multivariate polynomial spaces.

\subsection{Notation}
\noindent
In any given dimension $d \in \Nset$, we consider the domain $\support:=[-1,1]^d$, and
for some given real numbers $\paraone,\paratwo>-1$, the tensorized Jacobi measure 
\beu
d\rho= \otimes_{j=1}^d d\beta,
\eeu 
with density
\beu
\rho(y):=  \prod_{j=1}^d \beta(y_j),
\eeu
where
\beu
d\beta=\beta(t)dt:=c(1-t)^{\paraone}(1+t)^{\paratwo} dt, \quad c:=\left( \int_{-1}^{1} (1-t)^{\paraone}(1+t)^{\paratwo} dt \right)^{-1}.
\eeu

We may also consider the case $\support := [-1,1]^\Nset$ for which $d=+\infty$ and $d\rho$ is the Jacobi measure defined over $\support$ in the usual manner. We denote by $L^2(\support,d\rho)$ the Hilbert space of real-valued square-integrable functions with respect to $d\rho$ and denote by $\|\cdot\|$ the associated norm, i.e.  
\beu
\| v\|:= \left( \int_\Gamma |v(y)|^2 d\rho(y)\right)^{1/2}.
\eeu
Moreover, let $\cF$ be defined as the set $\Nset_0^d$, where $\Nset_0:=\{0,1,2,\dots\}$, in the case $d<+\infty$, or as the countable set of all finitely supported sequences from $\Nset_0^\Nset$ in the case $d=+\infty$. 
For any $\nu\in\cF$ we define 
\beu
\supp(\nu):=\{ \inddim \geq 1 : \nu_\inddim \neq 0 \},
\eeu
and for any multi-index set $\Lambda\subseteq\cF$ we define 
\beu
\supp(\Lambda):=\cup_{\nu \in \Lambda} \supp\{ \nu  \}.
\eeu
We say that a coordinate $y_j$ is {\it active} in the space $\P_\Lambda$ when $j\in \supp(\Lambda)$.

For the given real parameters $\paraone,\paratwo>-1$, we introduce the family $(J_n)_{n\geq 0}$ of univariate orthonormal Jacobi polynomials associated with the measure $d\beta$, and their tensorized counterpart
\beu
J_\nu(y)=\prod_{j\geq 1} J_{\nu_j}(y_j), \quad y=(y_j)_{j\geq 1},
\eeu
for any $\nu \in \cF$.
The $(J_{\nu})_{\nu \in \cF}$ are an $L^2(\Gamma,d\rho)$-orthonormal basis.
Particular instances of these polynomials are tensorized Legendre polynomials when $\paraone=\paratwo=0$ and tensorized Chebyshev polynomials of the first kind when $\paraone=\paratwo=-1/2$.

In the present paper we focus on finite multi-index sets $\Lambda$ which are downward closed in the sense of \eqref{dc}.
We also say that a polynomial space $\P_\Lambda$ is downward closed when it is associated with a downward closed multi-index set $\Lambda\subset \cF$. 
Recall that $\P_\Lambda$ has been defined as the span of the monomials $y\mapsto y^\nu$ for $\nu\in \Lambda$. Therefore it admits $(J_\nu)_{\nu\in \Lambda}$ as an $L^2(\Gamma,d\rho)$-orthonormal basis in the case of $\Lambda$ downward closed.
Sometimes we enumerate the indices $\nu$ using the lexicographical ordering, and denote this basis by $(\psi_k)_{k=1,\dots,n}$, where
\beu
n:=\#(\Lambda)=\dim(\P_\Lambda).
\eeu

Given a finite downward closed multi-index set $\Lambda\subset\cF$, we would like to approximate  the target function $\tarfun:\Gamma\to\Rset$ in the $L^2$ sense, using the noiseless evaluations $(\tarfun(y^i))_{i=1,\ldots,m}$ of $\tarfun$ at the points $(y^i)_{i=1,\ldots,m}$, where the $y^i$ are i.i.d.~random variables distributed according to $\rho$. We define the continuous $L^2$ projection of $\tarfun$ on the polynomial space $\P_\Lambda$ as 
\begin{equation*}
\Pi_\Lambda \tarfun := \argmin_{\genpol \in \P_\Lambda} \| \tarfun - \genpol \|, 
\end{equation*}
and the discrete least-squares approximation $\Pi_\Lambda^m \tarfun$ as   
\begin{equation}
\label{defPLambdan} 
\Pi_\Lambda^m \tarfun
:= \argmin_{\genpol \in \P_\Lambda} \| \tarfun - \genpol \|_m, 
\end{equation}
where we have used the notation
\beu
\| v \|_m := \left(\frac{1}{m}\sum_{i=1}^m v(y^i)^2\right)^{\frac{1}{2}}.
\eeu

We introduce the $m\times \#(\Lambda)$ design matrix $\bD$ and the vector $\boldsymbol{b}\in\Rset^m$ whose entries are given by $\bD_{i,k}=\psi_\indbasis(y^\indmeas)$ and $\boldsymbol{b}_{\indmeas}=\tarfun(y^\indmeas)$.
We define the Gramian matrix $\bG:=m^{-1} \bD^{\rm T} \bD$. The discrete least-squares projection in \eqref{defPLambdan} is then given by
\beu
\Pi_\Lambda^m \tarfun=\sum_{k=1}^{\#(\Lambda)}{\bf a}_k \psi_k,
\eeu
where the vector ${\bf a}=({\bf a}_k)\in \Rset^{\#(\Lambda)}$ is the solution to the normal equations
\beu
\bG \boldsymbol{\bf a} =  m^{-1}\bD^{\rm T} \boldsymbol{b}.
\eeu

\subsection{Previous results on the stability and accuracy of discrete least squares}
\noindent
We introduce the quantity
\begin{equation}
\label{defKVm}
K(\P_\Lambda) := \sup_{y\in\support} \sum_{\nu\in\Lambda} \left| J_\nu(y) \right|^2.
\end{equation}

It is proven in \cite{CCMNT2013} that discrete least squares in multivariate polynomial spaces are stable and accurate provided that a precise condition involving $m$ and $K(\P_\Lambda)$ is satisfied. 
Similar results have been proven in \cite{MNT2015} for the case of noisy observations of the target function, with several noise models. 

For any $\delta \in ]0,1[$ we introduce the positive quantity
\be
\zeta(\delta):=\delta + (1-\delta) \ln(1-\delta).
\label{zeta}
\ee
Given a threshold $\truncation\in\Rset_0^+$, we introduce the truncation operator  
\begin{align*}
T_\truncation(t):=& \textrm{sign}(t)\min\{ \truncation,\vert t \vert \}, \quad \textrm{ for any } t \in \Rset, 
\end{align*}
and use it to define the truncated discrete least-squares projection operator $u\mapsto T_\truncation(\Pi_\Lambda^m u)$.
The main results from \cite{CCMNT2013} concerning stability and accuracy of the discrete least-squares approximation with noiseless evaluations can be summarized as follows.

\begin{theorem}
\label{thm:theo_prev_analysis}
In any dimension $d$, for any $r > 0$, any $\delta\in ]0,1[$ and any downward closed multi-index set $\Lambda\subset\Nset_0^d$, one has
\be
\label{eq:stability_norm0}
\Pr \left( 
\left\{
 (1-\delta) \Vert \genpol \Vert^2 \leq
\Vert \genpol \Vert_m^2 
\leq 
(1+\delta) \Vert \genpol \Vert^2
, \ \forall \genpol \in \P_\Lambda     
\right\}
\right)\geq 1-2n\exp(-\zeta(\delta)m/K(\P_\Lambda)).
\ee
If the following condition between $m$ and $K(\P_\Lambda)$ is satisfied 
\begin{equation}
\dfrac{m}{\ln m} \geq \dfrac{(1+r)}{\zeta(\delta)} K(\P_\Lambda),
\label{eq:condition_K_M_CCMNT_jac_pol_noise}
\end{equation}
then 
\beu
\Pr \left( 
\left\{
 (1-\delta) \Vert \genpol \Vert^2 \leq
\Vert \genpol \Vert_m^2 
\leq 
(1+\delta) \Vert \genpol \Vert^2
, \ \forall \genpol \in \P_\Lambda     
\right\}
\right)\
 \geq 1 - 2 m^{-r}.
%\label{eq:stability_norm}
\eeu
Moreover, for any $\tarfun \in L^\infty(\support)$ with $\Vert \tarfun \Vert_{L^\infty(\support)} \leq \truncation$, the following holds:
\begin{align}
& \Pr \left( \Vert \tarfun - \Pi_\Lambda^m \tarfun \Vert  \leq \left(1+\sqrt{\dfrac{1}{1-\delta}}\right) 
\inf_{\genpol \in \P_\Lambda} \Vert \tarfun - \genpol \Vert_{L^\infty(\Gamma)}
\right) \geq 
1 - 2 m^{-r}
,
\label{eq:accuracy_noiseless_prob} \\
& \E\left( \Vert \tarfun -  T_\truncation(\Pi_\Lambda^m u)\Vert^2 \right) \leq \left(1+ \frac{4\zeta(\delta)}{(1+r) \ln m} \right) \Vert \tarfun - \Pi_\Lambda \tarfun \Vert^2 + 8 \truncation^2 m^{-r}.
\label{eq:accuracy_noiseless_exp}  
\end{align}
\end{theorem}

The above theorem states that under condition \eqref{eq:condition_K_M_CCMNT_jac_pol_noise} the discrete least-squares approximation is stable, since one has
\beu
 (1-\delta) \Vert \genpol \Vert^2 \leq
\Vert \genpol \Vert_m^2 
\leq 
(1+\delta) \Vert \genpol \Vert^2
, \ \forall \genpol \in \P_\Lambda     \Leftrightarrow  
 (1-\delta) \bI\leq \bG \leq (1+\delta) \bI,
\eeu
where $\bI$ denotes the identity matrix. 
Under the same condition, the discrete least-squares approximation is also accurate in probability, from \eqref{eq:accuracy_noiseless_prob}, and in expectation, from \eqref{eq:accuracy_noiseless_exp}, since the approximation error behaves like the best approximation error in $L^\infty$ or in $L^2$.

The quantity $K(\P_\Lambda)$ depends both on $\Lambda$ 
and on the chosen Jacobi measure, and therefore on the parameters $\theta_1,\theta_2$.
The following result from \cite{CCMNT2013} and \cite{M2014a} shows that,
once these two parameters are fixed, the quantity $K(\P_\Lambda)$ satisfies bounds that only depend on $\#(\Lambda)$, independently of the particular shape of $\Lambda$ and of the dimension $d$.

\begin{lemma}
\label{BoundJacobi}
In any dimension $d$ and for any finite downward closed set $\Lambda \subset \cF$, one has
\begin{equation*}
\#(\Lambda)
\leq
K(\P_\Lambda)
\leq
\begin{cases}
(\#(\Lambda))^{\ln 3 / \ln 2}, & \textrm{ if } \paraone=\paratwo=-1/2, \\
(\#(\Lambda))^{2\max\{\paraone,\paratwo\}+2}
& \textrm{ if } \paraone,\paratwo\in\Nset_0.
\end{cases}
\end{equation*}
\end{lemma}

Combining the two results, one therefore obtains sufficient conditions for stability and optimal accuracy expressed only in terms of a relation between $\#(\Lambda)$ and $m$. For example, in the case of the uniform measure that corresponds to $\theta_1=\theta_2=0$, this relation is of the form
\begin{equation*}
\dfrac{m}{\ln m} \geq c\, \left( \#(\Lambda) \right)^2, \quad c:=c(\delta,r).
\end{equation*}

\section{Optimal selection of downward closed polynomial spaces}
\label{sec:bestnterm_downward_closed}

The results recalled in the previous section hold for a given downward closed set $\Lambda\subset \cF$. We now consider the problem of optimizing the choice of $\Lambda$, or equivalently that of the space $\P_\Lambda$.

\subsection{Optimized index sets}

We define the family
\begin{equation*}
 \cM^d_n := \{\Lambda\subset \cF \; : \; \text{$\Lambda$ is downward closed and } \#(\Lambda)=n\},
%\label{eq:def_M_class}
\end{equation*}
of all downward closed sets of cardinality $n$ in $d$ coordinates.  Note that, in contrast to the family of {\it all} subsets of $\cF$ of cardinality $n$, the family $\cM^d_n$ is finite.

The error of best $n$-term polynomial approximation by downward closed sets is then defined by
\beu
\sigma_n(u):=\min_{\Lambda \in  \cM_{n}^{d}} \min_{ \genpol \in \P_\Lambda} \|\tarfun-\genpol\|.
\eeu
A best $n$-term downward closed set is a $\Lambda \in  \cM_{n}^{d}$ that achieves this minimum, that is, such that
\beu
 \Lambda^{opt} := \argmin_{\Lambda \in \cM_{n}^{d} } \min_{ \genpol \in \P_\Lambda} \|\tarfun-\genpol\|.
\eeu
Using Parseval identity, we find that $\Lambda^{opt}$ is also given by
\beu
 \Lambda^{opt} = \argmin_{\Lambda \in \cM_{n}^{d} } \sum_{\nu\notin \Lambda} |u_\nu|^2, \quad u_\nu=\int_\Gamma u(y) J_\nu(y) d\rho(y). 
\eeu
Note that such a set may not be unique due to possible ties in the values of the coefficients, in which case we consider a
unique choice by breaking the ties in some arbitrary but fixed way. We set
\be
u_n:=\Pi_{\Lambda^{opt} }u=\argmin_{ \genpol \in \P_{\Lambda^{opt}}} \|\tarfun-\genpol\|.
\label{un}
\ee
Of course, in the least-squares method, the discrete data do not allow us to identify $\Lambda^{opt}$. Instead, we rely on
\be
\Lambda^{opt}_{ m } := \argmin_{\Lambda \in \cM_{n}^{d} } \min_{ \genpol \in \P_\Lambda } \| \tarfun - \genpol \|_m,
\label{optim}
\ee
and compute
\beu
w_n:=\Pi_{\Lambda^{opt}_m}^m u=\argmin_{ \genpol \in \P_{\Lambda^{opt}_m}} \|\tarfun-\genpol\|_m.
\eeu
Our objective is now to compare the accuracy of the polynomial 
least-squares approximation based on $\Lambda^{opt}_{m}$ with the above optimal error $\sigma_n(u)$. For this purpose, we shall use the random variable
\begin{equation*}
  \costgen_{n}^{d} := \max_{\Lambda\in\cM_{n}^{d} } \max_{\genpol \in \P_\Lambda} \frac{\|\genpol\|^2}{\|\genpol\|^2_{m} }.
\end{equation*}

Note that the search of $\Lambda^{opt}_{m}$ remains a difficult task from the computational point of view, due to the fact that $\#(\cM_n^d )$ becomes very large even for moderate values of $n$ and $d$. As we discuss further, this cardinality also affects the conditions between $m$ and $n$ which guarantee the optimality of the least-squares approximation based on $\Lambda^{opt}_m$.

For this reason, it is useful to introduce an additional restriction on the potential index sets. We say that $\Lambda$ is {\it anchored} if and only if it is downward closed and satisfies in addition
\begin{equation*}
e_{\inddim} \in \Lambda \textrm{ and } \inddim' \leq \inddim \implies e_{\inddim'} \in \Lambda,
%\label{anchoredset}
\end{equation*}
where $e_\inddim$ and $e_{\inddim'}$ are the Kronecker sequences with $1$ at positions $\inddim$ and $\inddim'$, respectively. 
We also say that a polynomial space $\P_\Lambda$ is anchored when $\Lambda$ is anchored. Likewise, we define the family
\begin{equation*}
 \cA_{n}^{} := 
\{\Lambda\subset\cF \; : \; \text{$\Lambda$ is anchored}
\text{ and } 
\#(\Lambda)=n\}.
%\label{eq:def_A_class}
\end{equation*}

The restriction to anchored sets introduces an order of priority between the coordinates: given any $j\geq 1$, the coordinate $y_j$ is active in $\Lambda$ only if all the coordinates $y_k$ for $k<j$ are also active.
In particular, for any set $\Lambda\in\cA_n$ we have 
\be
\supp(\Lambda)=\{1,\dots,k\},
\label{supportanchored}
\ee 
for some $k\leq n-1$.

It is proven in \cite{CD2015} that, for relevant classes of parametric PDEs, the same algebraic convergence rates ${\cal O}(n^{-s})$ can be obtained when imposing the anchored structure on the optimally selected sets $(\Lambda_n)_{n \geq 1}$ with $\#(\Lambda_n)=n$. As we shall see further, one specific advantage of anchored sets is to completely remove the dependence on the dimension $d$ in the convergence analysis of the least-squares method, and allows us in particular to consider the infinite-dimensional framework.

Using the same notation as before with obvious modifications, we introduce the following entities:
\be
 \t \Lambda^{opt} := \argmin_{\Lambda \in \cA_{n} } \min_{ \genpol \in \P_\Lambda} \|\tarfun-\genpol\|, \qquad 
 \t u_n := \Pi_{\t\Lambda^{opt}} u=
  \argmin_{\genpol \in \P_{\t\Lambda^{opt}}} \| \tarfun - \genpol \|, 
  \label{tun}
\ee
\beu
\t  \Lambda^{opt}_{ m } := \argmin_{\Lambda \in \cA_{n} } \min_{ \genpol \in \P_\Lambda } \| \tarfun - \genpol \|_m, 
\qquad 
 \t w_n :=\Pi_{\t\Lambda^{opt}_m}^m u= \argmin_{ \genpol \in \P_{\t \Lambda^{opt}_m } } \| \tarfun - \genpol \|_m, 
\eeu
and
\begin{equation*}
\wt  \costgen_{n}^{} := \max_{\Lambda\in\cA_{n}^{}} \max_{\genpol \in \P_\Lambda} \frac{\|\genpol\|^2}{\|\genpol\|^2_{m} }.
\end{equation*}

\begin{remark}
\label{remmatrix}
The estimators $w_n$ and $\t w_n$ can be viewed as the solutions of a nonconvex 
optimization problem. This problem has a natural algebraic formulation. Recalling
$(J_\nu)_{\nu\in\cF}$ the $L^2(\Gamma,d\rho)$-orthonormal basis, we introduce 
for a given finite set $\cJ\subset \cF$ the
design matrix 
\be
\bD=(J_\nu(y^i)),
\label{design}
\ee 
where the row index $i$ ranges in $\{1,\dots,m\}$ and the column index $\nu$ ranges in $\cJ$.
Then, recalling the data vector $\bb=(u(y^i))_{i=1,\dots,m}$,
and taking $\cJ$ as the union of all downward closed sets of cardinality $n$, we find that the
component vector $\ba=(a_\nu)_{\nu\in\cJ}$ of $w_n=\sum_{\nu\in\cJ} a_\nu J_\nu$ is the solution
to the constrained minimization
\beu
\min \{ \|\bD \ba-\bb\|_{\ell^2} \;: \; \|\ba\|_{\ell^0_d}\leq n\}.
\eeu
Here $\|\ba\|_{\ell^0_d}$ is the cardinality of the smallest downward closed set $\Lambda\subset \cJ$
that contains all the nonzero entries of $\ba$.   In other words, we minimize over those $\ba$ 
whose support is contained in a downward closed set of cardinality at most $n$. 
Likewise, taking $\cJ$ as the union of all anchored sets of cardinality $n$, we find that the
component vector $\wt \ba=(\t a_\nu)_{\nu\in\cJ}$ of $\t w_n=\sum_{\nu\in\cJ} \t a_\nu J_\nu$ is the solution
of the constrained minimization
\beu
\min \{ \|\bD \wt \ba-\bb\|_{\ell^2} \;: \; \|\wt \ba\|_{\ell^0_a}\leq n\}.
\eeu
Here $\|\wt \ba\|_{\ell^0_a}$ is the cardinality of the smallest anchored set $\Lambda\subset \cJ$
that contains all the nonzero entries of $\wt \ba$.  It is well known that such optimization problems with $\ell^0$-type
constraints have combinatorial nature, and in turn numerical algorithms for computing their solutions
do not generally have polynomial complexity. This justifies in practice the use of convex relaxation methods 
such as basis pursuit, or greedy selection strategies such as orthonormal matching pursuit. 
While our paper is mainly directed towards the convergence properties of the estimators $w_n$ and $\t w_n$ which are the exact solutions to the above problems, we discuss these numerical methods
in the final section.
\end{remark}

We now would like to compare the estimators $w_n$ and $\t w_n$ with the best $n$-term approximation in the aforementioned classes of multi-index sets $\cM_{n}^d$ and $\cA_n^{}$. The following lemma shows the role played by $\costgen_{n}^{d}$ and $\wt \costgen_{n}^{}$ in quantifying the relation between the error  
achieved by the optimal discrete least-squares projection and  the error achieved by the optimal $L^2$ projection. 
\begin{lemma}
\label{thm:optimality_discrete_continuous_proj}
It holds that, for any $\Lambda\in \cM^d_n$,
\be
  \| \tarfun - w_{n} \| \le \| \tarfun - v \| + 2\sqrt{\costgen_{2n-1}^{d} } \| \tarfun - v \|_m,\quad v\in \P_\Lambda,
\label{estimlower}
\ee
and for any $\Lambda\in \cA_n$,
\be
  \| \tarfun - \t w_{n} \| \le \| \tarfun - \t v \| + 
2\sqrt{\wt \costgen_{2n -1}^{} } \| \tarfun - \t v \|_m, \quad  \t v \in \P_\Lambda.
\label{estimanchor}
\ee
\end{lemma}
\begin{proof}
Let $\Lambda\in \cM^d_n$ and define $\wh \Lambda := \Lambda \cup\Lambda^{opt}_m $. We first observe that $\wh \Lambda$ is 
also downward closed and 
$\#(\wh\Lambda)\leq 2n -1$ because any downward closed set contains the null multi-index.
Since  $w_n \in \P_{\Lambda^{opt}_m}$,  we have $v - w_n \in \P_{\wh\Lambda}$ for any $v \in \P_{\Lambda}$. 
It follows that
\begin{align*}
 \|v- w_n \| 
\le 
\sqrt{\costgen_{ 2n -1  }^{d} } \|v - w_n \|_m 
\le 
\sqrt{\costgen_{ 2n -1 }^{d} }(\| u- v \|_m + \| u - w_n \|_m ) \le 2\sqrt{\costgen_{ 2n -1 }^{d} } \| u - v \|_m,  
\end{align*}
and therefore
\beu
 \| \tarfun - w_n \| \le \| \tarfun - v  \| + \| v- w_n \| 
\le 
\| \tarfun - v \| + 2\sqrt{\costgen_{ 2n -1 }^{d} } \| \tarfun - v \|_m,
\eeu
which is \eqref{estimlower}. The proof of \eqref{estimanchor} is analogous.
\end{proof}

Note that the estimates in the above lemma imply in particular that
\be
  \| \tarfun - w_{n} \| \le \| \tarfun - u_n \| + 2\sqrt{\costgen_{2n-1}^{d} } \| \tarfun - u_n \|_m,
\label{estimlower1}
\ee
and
\beu
  \| \tarfun - \t w_{n} \| \le \| \tarfun - \t u_n \| + 
2\sqrt{\wt \costgen_{2n -1}^{} } \| \tarfun - \t u_n \|_m, 
\eeu
with $u_n$ and $\t u_n$ defined by \eqref{un} and \eqref{tun}. Note that they also imply
\be
  \| \tarfun - w_{n} \| \le \(1 + 2\sqrt{\costgen_{2n-1}^{d} }\)\| \tarfun - v \|_{L^\infty},\quad v\in \P_\Lambda,
\label{estimlower2}
\ee
for any $\Lambda\in \cM_n^d$, and
\be
  \| \tarfun - \t w_{n} \| \le \(1 + 
2\sqrt{\wt \costgen_{2n -1}^{} }\) \| \tarfun - \t v \|_{L^\infty}, \quad  \t v \in \P_\Lambda,
\label{estimanchor2}
\ee
for any $\Lambda\in \cA_n$.

\subsection{Probabilistic bounds}
\label{sec:estimation_C}
\noindent 
In view of Lemma ~\ref{thm:optimality_discrete_continuous_proj}, we are interested
in bounding the random variables $\costgen_{n}^{d}$ and $\wt \costgen_{n}^{}$.
In this section we give probabilistic bounds, which ensure that under certain conditions 
between $m$ and $n$, these random variables do not exceed a fixed value, here set to $2$, with high probability. 
In the whole section we choose $\delta=1/2$, so that, with the notation \eqref{zeta}, one has
\beu
\zeta:=\zeta(\delta)=\zeta(1/2)=(1-\ln 2)/2 \approx 0.153.
\eeu 

We define, for any $\nu\in\cF$, the ``rectangular'' set $\cR_\nu := \{ \mu\in\cF, \; \mu\le\nu\}$, 
and for any $n \geq 1$, the hyperbolic cross set  
\beu
  \cH_{n}^{d} := \left\{\mu\in\cF: \;\; \prod_{j=1}^d (\mu_j+1) \le n \right\}.
\eeu
Note that 
\beu
 \cH_{n}^{d} =
\bigcup_{\#(\cR_\nu) \leq n } \cR_\nu.
\eeu
The cardinality of $\cH_{n}^{d} $
is bounded by
\be
\label{eq:up_bound_hc}
  \#(\cH_{n}^{d}) \le n (1+\ln (n ))^{d-1},
\ee
see \cite[Appendix~A.2]{Migliorati2013} for a proof and some remarks on the accuracy of this upper bound. The relevance of the hyperbolic cross for our purposes is due to the following observation.

\begin{lemma}
The union of all downward closed sets of cardinality at most $n$ in finite dimension $d$ coincides with $\cH_{n}^{d}$, that is,
\be
  \bigcup_{ \Lambda \in \cM_{n}^{d} }\Lambda = \cH_{n}^{d}. 
  \label{unionlower}
\ee
\end{lemma}

\begin{proof}
On the one hand, all rectangles $\cR_\nu$ such that $\#(\cR_\nu) \leq n$ belong to $ \cM_{n}^{d}$, so that inclusion holds from right to left. On the other hand, inclusion from left to right follows by
observing that for any $\Lambda \in \cM_{n}^{d}$, one has $\Lambda=\cup_{\mu\in \Lambda} \cR_\mu$
and $\cR_\mu\subset  \cH_{n}^{d}$ for all $\mu \in \Lambda$.
\end{proof}

This leads us to a first probabilistic bound for the random variable $\costgen_n^d$. Indeed, using \eqref{unionlower} we obtain that 
\begin{align*}
  \Pr\left(\costgen_n^d > 2 \right) 
&= 
\Pr\left( \max_{\Lambda \in \cM_{n}^{d}} \max_{\genpol \in \P_\Lambda} \frac{\| \genpol \|^2}{\| \genpol \|_m^2} > 2 \right) 
\le 
\Pr\left(\max_{ \genpol \in\P_{\cH_{n}^{d}}} \frac{\| \genpol \|^2}{\| \genpol \|_m^2} >2 \right).
\end{align*}
Thus, using Theorem~\ref{thm:theo_prev_analysis} with $\delta=1/2$
combined with the estimates in Lemma \ref{BoundJacobi},
we obtain that, in any dimension $d$ and for any $r>0$, if $m$ and $n$ satisfy 
\begin{equation}
\label{condK2}
\frac m {\ln m} \geq
\begin{cases} 
\frac{(1+r)}{\zeta} (\#(\cH_{n}^{d}))^{\ln 3 / \ln 2},   &  \textrm{ with Chebyshev polynomials of the first kind}, 
\\
\frac{(1+r)}{\zeta} (\#(\cH_{n}^{d}))^{2\max\{\theta_1,\theta_2 \} +2 },   &
\textrm{ with Jacobi polynomials and } \theta_1,\theta_2 \in \Nset_0,
\end{cases}
\end{equation}
then
\beu
  \Pr(\costgen_{n}^{d}>2) \le 2m^{ - r }.
\eeu

From \eqref{unionlower} and \eqref{supportanchored} we also find that
the union of all anchored sets of cardinality at most $n$ satisfies the following inclusion
\beu
  \bigcup_{\Lambda \in \cA_n} \Lambda \subset \cH_n^{n-1}. 
\eeu
By similar arguments, we obtain the following probabilistic bound
for the random variable $\wt \costgen_{n}^{}$: in any dimension $d$,
for any $r>0$, if $m$ and $n$ satisfy 
\begin{equation}
\label{condK2_anch}
\frac m {\ln m} \geq
\begin{cases} 
\frac{(1+r)}{\zeta} (\#(\cH_{n}^{n-1 }))^{\ln 3 / \ln 2},   &  \textrm{ with Chebyshev polynomials of the first kind}, 
\\
\frac{(1+r)}{\zeta} (\#(\cH_{n}^{n-1}))^{2\max\{\theta_1,\theta_2 \} +2 },   &
\textrm{ with Jacobi polynomials and } \theta_1,\theta_2 \in \Nset_0,
\end{cases}
\end{equation}
then
\begin{equation*}
  \Pr(\wt \costgen_{n }^{}>2) \le 2m^{ - r }.
\end{equation*}

The above results describe the regimes of $m$ and $n$ such that
 $\costgen_{n}^{d}$ and $\wt \costgen_{n}^{}$
do not exceed $2$ with high probability. In the case
of downward closed sets, this regime is quite restrictive
due to the presence of $\ln(n)^{d-1}$ in the 
upper bound \eqref{eq:up_bound_hc} for the cardinality of $\cH_{n}^{d}$,
which enforces the sample size $m$ to be extremely large as $d$ grows.
Likewise, $m$ has to be extremely large compared to $n$ in the case
of anchored sets.

We next describe another strategy which yields similar probabilistic bounds
under less restrictive regimes. It is based on estimating the cardinality of $\cM_n^d$ and $\cA_n^{}$ and using union bounds.  
Our way of estimating $\#(\cM_{n}^{d})$ and $\#(\cA_n^{})$ is based on strategies for encoding any downward closed or anchored set. One first such strategy leads to the following result.

\begin{lemma}[First cardinality bound]
\label{lemmacard}
One has the cardinality estimates
\be
\#(\cM_{n}^{d}) \leq 2^{n d},
\label{cardMnd}
\ee
and
\be
\#(\cA_{n}^{}) \leq 2^{n(n-1)}.
\label{cardAn}
\ee
\end{lemma}

\begin{proof} We encode any downward closed set $\Lambda\subset \cF$ of cardinality $n$ in $d$ dimensions, 
by associating $d$ bits $b_{\nu,1},\dots,b_{\nu,d}$ to each multi-index $\nu\in\Lambda$,  
where the value of the $j$th bit is equal to one if $\nu + e_j \in \Lambda$, and 
equal to zero if $\nu + e_j \notin \Lambda$. We order these blocks of bits 
according to the lexicographic order $\nu^1,\dots, \nu^n$ of appearance of $\nu$ in $\Lambda$,
where $\nu^1=(0,\dots,0)$. The resulting bitstream
\beu
B_\Lambda := (b_{\nu^1,1},\dots ,b_{\nu^1,d},b_{\nu^2,1},\dots,b_{\nu^2,d},\dots,b_{\nu^n,1},\dots ,b_{\nu^n,d}),
\eeu
uniquely encodes $\Lambda$, that is, the encoding map $\Lambda\mapsto B_\Lambda$ is injective.
Indeed, assuming that $\nu^1,\dots,\nu^k$ have been already identified, then the partial bitstream
$(b_{\nu^1,1},\dots  ,b_{\nu^k,d})$ contains the information on the positions of all indices
$\nu\in\Lambda$ such that $\nu\notin\{\nu^1,\dots,\nu^k\}$ and $\nu-e_j\in\{\nu^1,\dots,\nu^k\}$
for some $j$. Therefore $\nu^{k+1}$ is uniquely determined as the index 
with smallest lexicographic order among such indices. 

Since the length of $B_\Lambda$ is $nd$, this model leads us to the upper bound
\eqref{cardMnd}. The same encoding can be applied to anchored sets of cardinality
at most $n$ with $n-1$ bits for each index, leading to \eqref{cardAn}, which also
directly follows from \eqref{cardMnd} and the fact that
$\#(\cA_{n}^{}) \leq \#(\cM_{n}^{n-1})$.
\end{proof}

\begin{remark}
Note that the cardinality estimate for the anchored sets is independent of the dimension $d$. In particular,
this allows us to derive some further results in the infinite-dimensional framework, when using
anchored sets.
\end{remark}

Recalling the definition of $K(\P_\Lambda)$ from \eqref{defKVm}, 
we introduce the following notation: 
\beu
K_n=\max_{\Lambda \in \cM_{n}^{d} } K(\P_\Lambda),
\eeu
\beu
\wt K_n=\max_{\Lambda \in \cA_{n}^{} } K(\P_\Lambda).
\eeu
Note that, according to Lemma \ref{BoundJacobi}, one has the estimate
\be
\wt K_n\leq K_n\leq 
\begin{cases}
n^{\ln 3 / \ln 2}, & \textrm{ if } \paraone=\paratwo=-1/2, \\
n^{2\max\{\paraone,\paratwo\}+2}
& \textrm{ if } \paraone,\paratwo\in\Nset_0.
\end{cases}
\label{boundjacobi2}
\ee
We are now in position to establish probabilistic bounds for the random variable $C^d_n$ and $\wt C_n$.

\begin{lemma}
In any finite dimension $d$, for any $r>0$, if $m$ and $n$ satisfy 
\be
\label{condK2_bitstreamgen}
\frac m {\ln m} \geq
\left(
1+r+\dfrac{ n d \ln 2}{\ln m} 
\right)   
\dfrac{K_n }{\zeta},  \ee
then 
\beu
\Pr\left( \costgen_{n}^{d}
>2 
  \right) \leq
2 n m^{-(r+1)}
\leq 2 m^{-r}.
\eeu
In any finite dimension $d$ or in infinite dimension,
for any $r>0$, if $m$ and $n$ satisfy 
\be
\label{condK2_bitstream_anchgen}
\frac m {\ln m} \geq
\left( 
1+r+\dfrac{  n^2  \ln 2}{\ln m} 
\right)
\dfrac{ \wt K_n}{\zeta},   
\ee
then 
\beu
\Pr\left(\wt \costgen_{n}
>2 
  \right) \leq
2 n m^{-(r+1)}\leq 2 m^{-r}.
\eeu
\end{lemma}

\begin{proof}
Recalling \eqref{eq:stability_norm0} and using a union bound, we obtain 
\begin{align*}
\Pr\left( \costgen_{n}^{d}   >2  \right) & = 
\Pr\left( \max_{\Lambda \in \cM_{n}^{d}} \max_{\genpol \in \P_\Lambda} \frac{\| \genpol \|^2}{\| \genpol \|_m^2} > 2 \right) 
\nonumber \\
&  \le 
\sum_{\Lambda \in \cM_n^d} 
\Pr\left(
\max_{ \genpol \in \P_{ \Lambda }} \frac{\| \genpol \|^2}{\| \genpol \|_m^2} >2 
\right) \nonumber  \\
& \le 
2^{n d} 
2 n 
\exp\left\{ - \zeta m /  K_n  \right\}  \nonumber \\
&  = 
2 n 
\exp\left\{ - \zeta m /  K_n    + n d  \ln(2)  \right\},
\end{align*}
where we have used the cardinality estimate \eqref{cardMnd}.
The final bound is smaller than $2n^{-r}$ under condition 
\eqref{condK2_bitstreamgen}.  The proof for $\wt C_n$ is completely similar,
using the cardinality estimate \eqref{cardAn}.
\end{proof}

\begin{remark}
Combining with \eqref{boundjacobi2}, we find that 
\eqref{condK2_bitstreamgen} holds if $r$, $d$, $m$ and $n$ satisfy 
\be
\label{condK2_bitstream}
\frac m {\ln m} \geq
\begin{cases} 
\left(
1+r+\dfrac{ n d \ln 2}{\ln m} 
\right)   
\dfrac{n^{\ln 3 / \ln 2} }{\zeta},   &  \textrm{ with Chebyshev polynomials of the first kind}, 
\\
\left(
1+r+\dfrac{ n d \ln 2 }{\ln m} 
\right)
\dfrac{
n^{2\max\{\theta_1,\theta_2 \} +2}
}{\zeta}
,  &
\textrm{ with Jacobi polynomials and } \theta_1,\theta_2 \in \Nset_0.
\end{cases}
\ee
Similarly, we find that \eqref{condK2_bitstream_anchgen}
holds 
if $r$, $m$ and $n$ satisfy 
\be
\label{condK2_bitstream_anch}
\frac m {\ln m} \geq
\begin{cases} 
\left( 
1+r+\dfrac{  n^2  \ln 2}{\ln m} 
\right)
\dfrac{ n^{\ln 3 / \ln 2}}{\zeta},   &  \textrm{ with Chebyshev polynomials of the first kind}, 
\\
\left(
1+r+\dfrac{  n^2 \ln 2 }{\ln m} 
\right)
\dfrac{ n^{2\max\{\theta_1,\theta_2 \} +2}  }{\zeta},  &
\textrm{ with Jacobi polynomials and } \theta_1,\theta_2 \in \Nset_0.
\end{cases}
\ee
\end{remark}

\noindent
We next give an improved result on the cardinality of $\cM_n^d$ and
$\cA_n$ based on another encoding strategy.

\begin{lemma}[Second cardinality bound]
\label{lemmacardimp}
One has the cardinality estimates
\be
\#(\cM_{n}^{d}) \leq  d^{n -1} (n-1)!,
\label{cardMndimp}
\ee
and
\be
\#(\cA_{n}^{}) \leq \left( \left( n-1 \right)! \right)^2.
\label{cardAnimp}
\ee
\end{lemma}

\begin{proof}
Given any downward closed multi-index set $\Lambda$ with $\#(\Lambda)=n$, we order the elements of $\Lambda$ in such a way that the set 
\beu
\Lambda^k:=
\{ 
\nu^1,\ldots,\nu^{k}
 \},
\eeu
obtained by retaining only the first $k$ elements of $\Lambda$, 
is downward closed for any $k=1,\ldots,n$. Such an ordering always exists,  
 and in general it is not unique. 
Notice that 
it always holds $\nu^1=(0,\dots,0)$. 
One way to impose a unique ordering is by taking for $\nu^k$ 
the smallest index $\nu$ in lexicographic order among those
$\nu\in \Lambda\setminus\{\nu^1,\dots,\nu^{k-1}\}$ such that 
$\{ \nu^1,\ldots,\nu^{k}\}$ is downward closed. Each $\nu^{k}$ can be uniquely 
characterized by choosing
some $l_k\in \{1,\dots,k-1\}$ and $j_k\in \{1,\dots,d\}$ such that
\beu
\nu^{k}=\nu^{l_k} + e_{j_k}.
%\label{nuk}
\eeu
Again this choice can be made unique by asking that $j_k$ is the
smallest number with such a property. Therefore, the resulting map
\beu
\Lambda \mapsto  (j_2,l_3,j_3,\dots,l_n,j_n),
\eeu
is well defined and injective. Hence 
\beu
\#(\cM_{n}^{d}) \leq  d (2 d )\cdots (n -1) d,
\eeu
which is \eqref{cardMndimp}. We obtain \eqref{cardAnimp}
from the inequality $\#(\cA_n)\leq \#(\cM_{n}^{n-1})$. 
\end{proof}

The above results lead to improved probabilistic bounds for the
random variables $C^d_n$ and $\wt C_n$.

\begin{lemma}
In any finite dimension $d$, for any $r>0$, if $m$ and $n$ satisfy 
\be
\label{condK2_bitstreamgenimp}
\frac m {\ln m} \geq
\left(1+r+
\dfrac{n
\ln(dn)}{\ln m}
\right) 
\dfrac{K_n }{\zeta},  \ee
then 
\beu
\Pr\left( \costgen_{n}^{d}
>2 
  \right) \leq
2 n m^{-(r+1)}
\leq 2 m^{-r}.
\eeu
In any finite dimension $d$ or in infinite dimension,
for any $r>0$, if $m$ and $n$ satisfy 
\be
\label{condK2_bitstream_anchgenimp}
\frac m {\ln m} \geq
\left( 
1+r+\dfrac{  2n  \ln n}{\ln m} 
\right)
\dfrac{ \wt K_n}{\zeta},   
\ee
then 
\beu
\Pr\left(\wt \costgen_{n}
>2 
  \right) \leq
2 n m^{-(r+1)}\leq 2 m^{-r}.
\eeu
\end{lemma}

\begin{proof}
Recalling \eqref{eq:stability_norm0}  
and using the inequality $n ! \leq e \sqrt{n} (n/e)^{n}$
which holds for any $n\geq 1$, we obtain by means of a union bound  
that, for any $n\geq 2$, 
\begin{align*}
\Pr\left( \costgen_{n}^{d}
 >2 
\right) 
& \leq 
d^{n -1} (n-1)!
\ 
(2 n)
\
\exp\left\{ - \zeta m /  K_n  \right\}
\nonumber \\
& \leq d^{n -1} 
e \sqrt{n-1} \left( \dfrac{n-1}{e}\right)^{n-1}
\ 
(2 n) 
\
\exp\left\{ - \zeta m /  K_n  \right\}
\nonumber \\
& 
=
2 n
\exp\left\{ - \zeta m /  K_n   + 
(n-1/2) 
\ln\left(
\dfrac{
d (n-1)
}{e} 
     \right)
     - \dfrac{1}{2} \ln \left( \dfrac{d}{e} \right)  
     + 1
\right\}
\nonumber\\
& 
\leq 
2 n
\exp\left\{ - \zeta m /  K_n   + 
n
\ln(dn)
\right\},
\nonumber
\end{align*}
where we have used the cardinality estimate \eqref{cardMndimp}.
The final bound is smaller than $2n^{-r}$ under condition 
\eqref{condK2_bitstreamgenimp}.  
Trivially the final bound holds true also when $n=1$. 
The proof for $\wt C_n$ is completely similar,
using the cardinality estimate \eqref{cardAnimp}.
\end{proof}

\begin{remark}
Combining with \eqref{boundjacobi2}, we find that 
\eqref{condK2_bitstreamgenimp} holds if $r$, $d$, $m$ and $n$ satisfy 
\be
\label{condK2_bitstreamimp}
\frac m {\ln m} \geq
\begin{cases} 
\left(
1+r+\dfrac{ n \ln(dn)}{\ln m} 
\right)   
\dfrac{n^{\ln 3 / \ln 2} }{\zeta},   &  \textrm{ with Chebyshev polynomials of the first kind}, 
\\
\left(
1+r+\dfrac{ n  \ln (dn) }{\ln m} 
\right)
\dfrac{
n^{2\max\{\theta_1,\theta_2 \} +2}
}{\zeta}
,  &
\textrm{ with Jacobi polynomials and } \theta_1,\theta_2 \in \Nset_0.
\end{cases}
\ee
Similarly, we find that \eqref{condK2_bitstream_anchgenimp}
holds 
if $r$, $m$ and $n$ satisfy 
\be
\label{condK2_bitstream_anchimp}
\frac m {\ln m} \geq
\begin{cases} 
\left( 
1+r+\dfrac{  2n  \ln n}{\ln m} 
\right)
\dfrac{ n^{\ln 3 / \ln 2}}{\zeta},   &  \textrm{ with Chebyshev polynomials of the first kind}, 
\\
\left(
1+r+\dfrac{  2n\ln n }{\ln m} 
\right)
\dfrac{ n^{2\max\{\theta_1,\theta_2 \} +2}  }{\zeta},  &
\textrm{ with Jacobi polynomials and } \theta_1,\theta_2 \in \Nset_0.
\end{cases}
\ee
\end{remark}

The regimes of $m$ and $n$ described
by the above results are in principle less restrictive than 
those previously obtained using the cardinality of $\cH_{n}^{d}$ or $\cH_{n}^{n-1}$.
Indeed, for most regimes of $n$ and $d$ the cardinalities $\#(\cH_{n}^{d})$ or $\#(\cH_{n}^{n-1})$ 
are larger than $n\ln(dn)$ and $n\ln(n)$, respectively.
We may summarize the probabilistic bounds established in this section 
as follows: for any $r>0$ and any 
$n\geq 1$
one has $C_n^d\leq 2$ with probability larger than $1-2m^{-r}$
provided that \eqref{condK2} 
or \eqref{condK2_bitstream} or 
\eqref{condK2_bitstreamimp} holds. Likewise, one has $\wt C_n\leq 2$ with probability larger than $1-2m^{-r}$
provided that \eqref{condK2_anch}
or \eqref{condK2_bitstream_anch}
or \eqref{condK2_bitstream_anchimp}
holds.

\begin{remark}
The encoding strategies that are used for proving Lemma \ref{lemmacard} 
and Lemma \ref{lemmacardimp} are both redundant, leading to an overestimation
of $\#(\cM_n^d)$ and $\#(\cA_n)$. We are not aware of estimates for
these cardinalities which are provably sharp up to multiplicative constants
independent of $n$ and $d$. However, we can establish lower bounds
which show that for certain particular regimes of $n$ and $d$, these cardinalities grow 
exponentially fast. One
simple instance of a lower bound for $\#(\cM_n^d)$  in the regime where $n-1\leq d$ is obtained
as follows: we note that $\cM_n^d$ 
includes in particular all sets of the form $\{(0,\ldots,0)\}\cup \{e_j\,:\, j\in S\}$ for $S\subset \{1,\dots,d\}$
such that $\#(S)=n-1$. It follows that 
\beu
\#(\cM_n^d)\geq {d\choose n-1}.
\eeu
In the regime where $n=d/2$ (for even $d$), we thus find that
$\log_2(\#(\cM_n^d))$ grows at least as fast as $d$.
\end{remark}

\begin{remark}
\label{remcom}
It is interesting to compare the probabilistic bounds obtained in this section with 
Restricted Isometry Properties (RIP) recently obtained in \cite{CDTW}, that are a common vehicle
in the analysis of compressed sensing schemes.  Recalling the design matrix $\bD$
introduced in \eqref{design}, and defining its renormalized version $\Phi:=m^{-1/2} \bD$, we indeed see that the property $C_n^d \leq 2$ can be rephrased as
\beu
\frac 1 2 \|\ba\|^2 \leq \|\Phi \ba\|^2, \quad   \|\ba\|_{\ell^0_d}\leq n,
\eeu
that is, for all $\ba$ whose support is contained in a downward closed set of cardinality at most $n$.
In \cite{CDTW}, it is shown that the RIP property
\be
(1-\delta) \|\ba\|^2 \leq \|\Phi \ba\|^2\leq (1+\delta) \|\ba\|^2, \quad   \|\ba\|_{\ell^0_d}\leq n,
\label{rip}
\ee
holds with probability at least $1-\gamma$ if 
\be
m\geq 2^6 e\frac {K_n}{\t \delta^2}\ln\(\frac {K_n}{\t \delta^2}\)\max\Big\{\frac{2^5}{\t \delta^4}\ln\(40\frac {K_n}{\t \delta^2}\ln\(\frac {K_n}{\t \delta^2}\)\)\ln(4N),
\frac 1 {\t \delta}\ln\(\frac{1}{\gamma\t \delta}\ln\(\frac {K_n}{\t \delta^2}\)\) \Big\},\quad \t \delta:=\frac {\delta}{13},
\label{newregime}
\ee
where $N=\#(\cJ)$ with $\cJ$ the union of all downward closed sets of cardinality at most $n$,
that is, $\cJ=\cH_n^d$. 
Note that in our analysis, we are only interested in establishing the lower inequality
in the RIP, with the particular constant $\delta=\frac 1 2$. However, since 
we have started from the two-sided estimate in \eqref{eq:stability_norm0}, one easily checks that 
the same analysis leads to the validity of the RIP property \eqref{rip} 
with probability at least $1-\gamma$, under the 
regime
\be
\label{condK2_altbitstream_lowerdelta} 
\frac m {\ln m} \geq
\left(1+ \frac {\ln (2/\gamma)}{\ln m}+
\dfrac{n
\ln(dn)}{\ln m}
\right) 
\dfrac{K_n}{\zeta(\delta)},
\ee
where $\zeta(\delta)$ is again given by \eqref{zeta}.
From an asymptotic point of view, it can be checked that the regime \eqref{newregime} is more favorable than \eqref{condK2_altbitstream_lowerdelta}: indeed $n$ appears on the right side of  \eqref{newregime} only through $K_n$ up to logarithmic factors, while it appears through $nK_n$ in the right side of \eqref{condK2_altbitstream_lowerdelta}. However, the multiplicative constant on the right side of \eqref{newregime} is prohibitively large: for example in the case $\delta=\frac 1 2$ which is considered in the present paper, we find that $\t \delta=\frac 1 {26}$ and therefore $2^6 e\frac {1}{\t \delta^2}\frac{2^5}{\t \delta^4} > 10^{12}$. 
This shows that the regime described by \eqref{condK2_altbitstream_lowerdelta} is more favorable for moderate values of $n$. Similar remarks hold when considering anchored sets rather than downward closed sets. 
\end{remark}

\subsection{Accuracy of the optimized discrete least-squares approximation}
\label{sec:estimationerror}
\noindent 
We are now in position to state our main results concerning the accuracy of the discrete least-squares approximation $w_n$ and $\t w_n$ over the optimized index set $\Lambda^{opt}_{m}$ and $\t  \Lambda^{opt}_{ m }$. 
These results show that the accuracy compares favorably with the best approximation error of the function $u$, measured either in $L^\infty$ or $L^2$, using optimal choices of downward closed or anchored sets (which might differ from the sets $\Lambda^{opt}_{m}$ and $\t  \Lambda^{opt}_{ m }$). We begin with a result expressed in probability.

\begin{theorem} 
\label{theoprob}
Consider a function $u$ defined on $\Gamma$ and let $r>0$.
In any finite dimension, 
under condition \eqref{condK2}, or 
\eqref{condK2_bitstream}, or 
\eqref{condK2_bitstreamimp}, with $n$ replaced by $2n-1$,
it holds that 
\begin{align*}
\Pr\left(  
\| u -  w_n \|
\leq 
(1+2\sqrt2)  \min_{\Lambda\in \cM^d_n}\min_{v \in \P_{\Lambda} } \Vert u - v \Vert_{L^\infty(\Gamma)}    \right)   
\geq 1 - 2 m^{-r}.
\end{align*}
In any finite or infinite dimension, 
under condition \eqref{condK2_anch}, 
or \eqref{condK2_bitstream_anch}, 
or \eqref{condK2_bitstream_anchimp},
with $n$ replaced by $2n-1$, 
it holds that 
\begin{align*}
\Pr\left(  
\| u -  \t w_n \|
\leq 
(1+2\sqrt2)  \min_{\Lambda\in \cA_n}\min_{v \in \P_{\Lambda} }\Vert u - v \Vert_{L^\infty(\Gamma)}    \right)   
\geq 1 - 2 m^{-r}.
\end{align*}
\end{theorem}
\begin{proof}
These estimates immediately follow from 
\eqref{estimlower2} and \eqref{estimanchor2} combined with the probabilistic bounds
from the previous section. 
\end{proof}

We next give a result expressed in expectation for the truncated
discrete least-squares projection $T_\truncation(w_n)$ and $T_\truncation(\t w_n)$.

\begin{theorem}
\label{theoexp}
Consider a function $u$ defined on $\Gamma$,
such that $|u(y)|\le \truncation$ for any $y\in\support$, and let $r>0$.
In any finite dimension,
under condition \eqref{condK2}, or 
\eqref{condK2_bitstream}, or 
\eqref{condK2_bitstreamimp}, with $n$ replaced by $2n-1$, 
it holds that 
\begin{align}
&
  \E( \| u-   
  T_\truncation(w_n)
  \|^2) \le 
(9 + 4\sqrt{2})
\| u - u_n \|^2 + 8  \truncation^2 m^{ - r }. 
\label{eq:exp_accuracy_opt}
\end{align}
In any finite or infinite dimension, 
under condition \eqref{condK2_anch}, 
or \eqref{condK2_bitstream_anch}, 
or \eqref{condK2_bitstream_anchimp},
with $n$ replaced by $2n-1$,
it holds that 
\begin{align}
&
  \E( \| u-   
  T_\truncation(\t w_n)
  \|^2) \le 
(9 + 4\sqrt{2})
\| u - \t u_n \|^2 + 8  \truncation^2 m^{ - r }.
\label{eq:exp_accuracy_opt_anch}
\end{align}
\end{theorem}
\begin{proof}
For \eqref{eq:exp_accuracy_opt}, we distinguish between the two complementary events 
$\Omega_1:=\{C_{2n-1}^d\leq 2\}$ and $\Omega_2:=\{C_{2n-1}^d> 2\}$ and write
\beu
 \E( \| u-   
  T_\truncation(w_n)
   \|^2) = \E( \| u-   
  T_\truncation(w_n)
    \|^2\, | \, \Omega_1) \Pr(\Omega_1)+\E( \| u-   
  T_\truncation(w_n)
    \|^2\, | \, \Omega_2) \Pr(\Omega_2)=:E_1+E_2.
\eeu
Since $|u-T_\truncation(w_n)|\leq 2\truncation$ and $\Pr(\Omega_2)\leq 2m^{-r}$, the second term $E_2$
is bounded by $8  \truncation^2 m^{ - r }$. For the first term $E_1$, we combine \eqref{estimlower1}
and the fact that $|u-T_\truncation(w_n)|\leq |u-w_n|$ to obtain the bound
\begin{align*}
E_1 
\leq & \E \((\| \tarfun - u_n \| + 2\sqrt{2} \| \tarfun - u_n \|_m)^2\) 
\\
= & \| \tarfun - u_n \|^2  + 4 \sqrt{2} \| \tarfun - u_n \| \E \(  \| \tarfun - u_n \|_m  \) + 8 \E \( \| \tarfun - u_n \|_m^2 \)
 \\
\leq & (9 + 4\sqrt{2}) \| \tarfun - u_n \|^2.
\end{align*}
The proof of \eqref{eq:exp_accuracy_opt_anch} is analogous.
\end{proof}

\begin{remark} 
The constants $1+2\sqrt 2$ and $9+4\sqrt2$
in the above theorems can be reduced if one further restricts the
regime between $m$ and $n$ so that $C_{2n-1}^d$ and $\t C_{2n-1}$ are close
to $1$ with high probability.
\end{remark}

\section{Alternative algorithms for model selection}

The actual computation of the optimized discrete least-squares approximations $w_n$ and $\tilde w_n$ in Theorems \ref{theoprob} and \ref{theoexp} would require an exhaustive search among all possible choices of downward closed or anchored sets of a given cardinality $n$, and this task might become computationally prohibitive when $n$ and $d$ are simultaneously large. Our results should therefore mainly be viewed as a benchmark in arbitrary dimension $d$ for assessing the performance of fast selection algorithms.

We review hereafter other strategies, which could be used for the purpose of selecting a proper polynomial space, and relate them with the results from this paper.
\begin{itemize}
\item
Model selection by complexity regularization: the above discussed least-squares 
methods on optimized downward closed or anchored sets may be viewed 
as an instance of a model selection procedure.
Model selection is a widely studied topic in statistical learning theory, in various
settings such as classification, density estimation, denoising, regression 
and inverse problems. In the regression framework, a typical approach consists 
in adding a complexity penalization term to the least-squares empirical risk 
to be minimized, see for example Chapter 12 in \cite{Gyorfi2002}.
One general result for such estimators is provided by Theorem 12.1 
of \cite{Gyorfi2002} established in the bounded regression framework. It gives
an oracle bound which has the typical form of the minimum
over all considered models of the sum of the approximation error 
and of a penalty term which is always larger than $m^{-1}$,
and which persists even as the noise level tends to zero. Therefore we cannot
retrieve from these standard results for model selection  the potentially fast
convergence rates which can be derived from the above Theorem \ref{theoexp} 
in the noisefree case. Let us note that, from a computational point of view,
this approach has the same complexity as the one required to compute
the estimators $w_n$ or $\t w_n$ in the present paper,
since it is based on an exhaustive optimization over the index sets. It is 
therefore prohibitive for large values of $d$ and $n$.
\item
Convex relaxation of $\ell^0$ problems: as explained in Remark \ref{remmatrix},
the estimators $w_n$ and $\t w_n$ are solutions to nonconvex optimization
problems of $\ell^0$ type up the additional prescription of the downward closed
or anchored structure of the index sets. Convex relaxation methods based on 
$\ell^1$ or weighted-$\ell^1$ minimization, such as basis pursuit or LASSO,
have been intensively studied, in particular under RIP properties for the design
matrix. In the context of multivariate approximation, they have been first studied
in the trigonometric polynomial framework \cite{CDFR}, and then in 
the algebraic polynomial framework \cite{RS2016,RW2016,CDTW}.
The corresponding optimization algorithms are computationally much less intensive than
the complete optimization of the index set which is needed to compute $w_n$ or $\t w_n$,
yet their complexity is still polynomial in the cardinality of the
set $\cJ$ that indexes the columns of $\bD$. This set coincides with the hyperbolic cross
$\cH^d_n$ in the downard closed case, leading to a computational
cost that could still be prohibitive for simultaneously large values of $n$ and $d$.
Note that these methods do not necessarily generate downward closed or anchored sets.
While one may use the compressed sensing theory based on RIP properties
to analyze the accuracy of the resulting estimators, see in particular the recovery
guarantee established in Theorem 4.7 of \cite{CDTW},
it is not clear to us that they achieve the same optimality bounds as described by 
Theorem 2 and 3. 
\item
Greedy algorithms: one classical alternative to the above described 
convex relaxation methods are greedy algorithms such as orthonormal
matching pursuit (OMP) and its variants, such as iterative hard or soft thresholding. 
Convergence bounds for the
estimators produced by theses algorithms have been
established under RIP properties, see \cite{CDD,TZ} for OMP
in a general framework, and Theorem 4.9 of \cite{CDTW}
for iterative hard thresholding in the context of multivariate polynomial approximation.
Similar to convex relaxation methods, it is not clear that the estimators
obtained by these approaches achieve the same optimality bounds as described by 
Theorem 2 and 3. From a computational
point of view, the complexity of each step of OMP scales
linearly in the cardinality of $\cJ$ leading to a smaller computational
cost than convex relaxation methods, yet that could still be prohibitive 
for simultaneously large values of $n$ and $d$. One natural way to reduce the complexity
is to enforce the downward closed or anchored structure in the 
index selection: if $\Lambda_k$ is the index set generated after $k$ steps of OMP,
we construct $\Lambda_{k+1}=\Lambda_k\cup\{\nu\}$ by maximizing the 
inner product of the residual with the colums of $\bD$ corresponding only to the
indices $\nu$ such that the set $\Lambda_k\cup\{\nu\}$ remains downward closed (or anchored).
This restriction has the effect of significantly reducing the complexity. However 
it is not clear that any recovery guarantee can be established for such an algorithm.
\item
Relaxed minimization: let us observe that one could replace the set
$\Lambda^{opt}_{m}$ defined in \eqref{optim} with a near-optimal set $\Lambda^{near}_{m}$ 
in the sense that one has
\beu
 \min_{ \genpol \in \P_{\Lambda^{near}_{m}}} \| \tarfun - \genpol \|_m \leq 
 C \min_{\Lambda \in \cM_{n}^{d} }\min_{ \genpol \in \P_\Lambda } \| \tarfun - \genpol \|_m,
\eeu
for some fixed constant $C\geq 1$. Then, it is easily checked that similar convergence bounds
can be established for the resulting estimators, up to changing the multiplicative constants.
If one is only interested in establishing convergence rates, 
the optimality criterion can be even further relaxed by only asking that 
\beu
 \min_{ \genpol \in \P_{\Lambda^{near}_{m}}} \| \tarfun - \genpol \|_m \leq 
 C \min_{\Lambda \in \cM_{\xi n}^{d} }\min_{ \genpol \in \P_\Lambda } \| \tarfun - \genpol \|_m,
\eeu
for some fixed $0<\xi\leq 1$. However, designing a fast selection algorithm that would produce 
such near-optimal sets is currently an open problem. A similar objective has been accomplished in the setting of tree structured index sets,
see \cite{BD}.
\end{itemize}

\bibliographystyle{plain}

\end{document}